\documentclass[11pt,twoside]{amsart}
\usepackage{amsmath,amssymb,amsthm}

\newtheorem{thm}{Theorem}[section]

 \newtheorem{lem}{Lemma}[section]
 \newtheorem{prop}{Proposition}[section]
 
\newtheorem{rem}{Remark}[section]

\def\Id{{\rm Id}\,}
\def\d{\partial}
\def\ddj{\dot \Delta_j}

\def\tilde{\widetilde}

\def\wt{\widetilde}

\newcommand{\bFormula}[1]{\begin{equation} \label{#1}}
\newcommand{\eF}{\end{equation}}

\newcommand{\Div}{{\rm div}_x}
\newcommand{\Grad}{\nabla_x}
\newcommand{\vr}{\varrho}
\newcommand{\bTheorem}[1]{\begin{Theorem} \label{T#1} }
\newcommand{\eT}{\end{Theorem}}
\newcommand{\bProposition}[1]{\begin{Proposition} \label{P#1}}
\newcommand{\eP}{\end{Proposition}}
\newcommand{\bLemma}[1]{\begin{Lemma} \label{L#1} }
\newcommand{\eL}{\end{Lemma}}
\newcommand{\bCorollary}[1]{\begin{Corollary} \label{C#1} }
\newcommand{\eC}{\end{Corollary}}

\newcommand\R{\mathbb{R}}

\newcommand\Z{\mathbb{Z}}

\newcommand{\ep}{\varepsilon}

\renewcommand{\div}{\mbox{\rm div}\;\!}

\def\dr{\delta\!\rho}
\def\dw{\delta\!w}
\def\df{\delta\!\phi}

\def\cC{{\mathcal C}}

\def\cF{{\mathcal F}}

\newcommand{\Int}{\displaystyle \int}

\newcommand{\with}{\quad\hbox{with}\quad}
\newcommand{\andf}{\quad\hbox{and}\quad}

\begin{document}

\title[Compressible Euler system
with various couplings]{The global existence issue  for the compressible Euler system
with Poisson or Helmholtz  couplings}
\author{X. Blanc, R. Danchin, B. Ducomet and \v{S}. Ne\v{c}asov\'a}
\begin{abstract}
\noindent We consider the Cauchy problem for the barotropic Euler system coupled to Helmholtz or Poisson equations, in the whole space.
We assume that the initial density is small enough, and that the initial velocity is  close to some
reference vector field $u_0$ such that the spectrum of $Du_0$ is bounded away from zero. We prove the existence of a global unique
solution with  (fractional) Sobolev regularity, and algebraic time decay estimates. 
Our work extends the papers  by D. Serre and M. Grassin  \cite{GS,G2,S}
dedicated to  the compressible Euler system without coupling and  integer regularity exponents. 
\end{abstract}
\maketitle


\section{Introduction}

It is well-known from physics textbooks that Euler-Poisson system 
is relevant either  {\it in the attractive case}  for  describing plasma dynamics when the compressible electron fluid interacts with a constant charged ionic background \cite{CH},
 or {\it in the repulsive case} as a model of evolution of self-gravitating gaseous stars \cite{Ch}.

The Euler-Helmholtz system has been  introduced more recently to investigate some vortex flows of compressible media \cite{CC}.
\vskip0.25cm
In the plasma context, the inertia of ions being much larger than that of electrons, a standard approximation
consists in assuming that the dynamics reduces to the motion of electrons in a constant ionic background. A simple description is given by the following barotropic Euler system:
\bFormula{i1p}
\partial_t n + \Div (n u) = 0,
\eF
\bFormula{i2p}
\partial_t (n u) + \Div (n u \otimes u) + \frac{1}{m_e}\Grad P(n)=\frac{ne}{m_e}\Grad \phi,
\eF
with initial data
\bFormula{i0p}
(n,u)(x,0)=(n_0,u_0)(x),
\eF
 where the electric potential $\phi$ satisfies the Helmholtz equation
\bFormula{i3p}
\Delta \phi-\mu^2 \phi=4\pi e (n-\overline n)\with 
|\phi|\to 0\ \hbox{ as }\ |x|\to\infty.
\eF
 Here, $n(t,x)$ is the density of electrons of charge $e$ and mass $m_e$
($t$ denotes the time variable and $x\in\R^3,$ the Eulerian spatial variable), $\mu$ is a non negative parameter, $\overline n$ is the density
 of the ionic background, $u(t,x)$ is the velocity of the electronic flow and $P(n)=A n^\gamma$ $(\gamma\geq1$) is the electronic pressure.    
\vskip0.25cm
 In the astrophysical context, as explained in e.g. \cite{Ch} \cite{Chi}, a simple evolutionary 
 model for the density $\rho = \rho (t,x)$, the velocity
field $u = u(t,x)$ and the gravitational potential $\phi=\phi(t,x)$ 
 is the following self-gravitating Euler-Helmholtz system: 
\bFormula{i1g}
\partial_t \rho + \Div (\rho u) = 0,
\eF
\bFormula{i2g}
\partial_t (\rho u) + \Div (\rho u \otimes u) + \Grad\Pi(\rho)=-\vr\Grad \phi,
\eF
\bFormula{i3g}
\Delta \phi-\mu^2\phi=4\pi \Gamma\rho,
\eF
with initial data
\bFormula{i0g}
(\rho,u)(x,0)=(\rho_0,u_0)(x),
\eF
where $\Pi(\rho)=A\rho^\gamma$ is the barotropic pressure with $A>0$ and the adiabatic exponent $\gamma>1,$ and $\Gamma$ 
is the Newton  gravitational constant.
\vskip0.25cm
In  the attractive case (that is, \eqref{i1g}, \eqref{i2g}, \eqref{i3g}), as the gravitational ingredient in the system (Newton equation) must produce a convergent potential, at least in the massless case ($\mu=0$), the density has to be decreasing at large distance and a reasonable model for a gaseous star is thus a compactly supported density materializing the domain of the star matter.
This  introduces the classical difficulty of vacuum (as first observed by Kato \cite{K}) when  symmetrizing the system. Despite this,  the corresponding Cauchy problem for Euler-Poisson with vacuum for strong solutions was solved locally in time in the eighties by various authors, among them:
 Makino  \cite{Ma,Ma2}, Makino-Ukai \cite{MU}, Makino-Pertame \cite{MP},  Gamblin \cite{ga}, B\'ezard \cite{be}, Braun and Karp  \cite{BK}
(see also \cite{Ma3} for a clear survey). 

Owing  to the fact that no dissipative process takes place in the system,
existence results are expected to be only local in time even for small data \cite{CW}
 (see blow-up results of Chemin \cite{C} (3D case) or Makino and Perthame \cite{MP} (1D spherically symmetric case)).
 However,   in a series of papers \cite{S} \cite{GS} \cite{G2}, D. Serre and M. Grassin pointed 
 out  that under a suitable   ``dispersive'' spectral 
 condition  on the initial velocity that will be specified in the next section,
 and a smallness hypothesis on the initial 
density, the compressible Euler system admits a unique global smooth solution. 
They first considered the isentropic case for a compactly supported density, then relaxed the support condition and finally extended the result to the non-isentropic case. In \cite{LM}, L\'ecurieux-Mercier 
obtained similar results for  the Van der Waals equation of state. 
\smallbreak
The main purpose of the present work is to show that the Serre-Grassin global existence result extends to  the compressible Euler system coupled with the  Poisson or Helmholtz equations, 
\emph{regardless of the sign of the coupling}. 

Compared to the pure Euler system, once the global existence problem has been solved, 
a new situation   arises regarding the study of the static solutions of the Euler-Poisson (resp. Euler-Helmholtz) system and their stability. In the gravitational Poisson case, this is a classical question
in astrophysics (see \cite{C} \cite{SP}) which is difficult to address in the general case for technical reasons (see \cite{Li} for a survey  in the 1D case).
\bigbreak
The rest of the paper is structured as follows. 
In the next section, we state our main results and give some insights on our strategy. 
In Section \ref{aux}, we establish decay estimates in Sobolev spaces first for 
the multi-dimensional Burgers equation (that is expected to provide 
an approximate solution for our system), and next for  the compressible 
Euler equation coupled with the Poisson or the Helmholtz equation.
The next section is devoted to the proofs of  the main global existence 
results, then we show the uniqueness of the solution.
Some technical results like, in particular, first and second order commutator 
estimates are proved in the appendix. 
\medbreak
Throughout the paper, $C$ denotes a harmless `constant' that may change from line
to line, and we use sometimes the notation $A\lesssim B$ to mean that 
$A\leq CB.$ The notation $A\approx B$  is used if  both $A\lesssim B$ and $B\lesssim A.$



\section{Main results}\label{main}

In order to reformulate Systems \eqref{i1p}-\eqref{i3p} and 
\eqref{i1g}-\eqref{i0g} in a unified way,  let us  consider 
the following general system 
for the density $\vr=\vr(t,x),$ velocity field $u=u(t,x)$ and potential $\phi=\phi(t,x)$:
\bFormula{i1}
\partial_t \vr + \Div (\vr u) = 0,
\eF
\bFormula{i2}
\partial_t (\vr u) + \Div (\vr u \otimes u) + \Grad p(\vr)=\kappa\vr\Grad \phi,
\eF
\bFormula{i3}
\Delta \phi-\mu^2 \phi=4\pi G\vr,
\eF
supplemented with initial data
\bFormula{i0}
(\vr,u)(0,x)=(\vr_0,u_0)(x).
\eF
That system encompasses  the previous ones: 
take  $\bar n=0,$ $\vr(t,x) = n(t,x),$    $p(\vr)=\frac{1}{m_e}P(n),$  $G=e$
and  $\kappa=e/m_e$ (resp. $\vr(t,x) = \rho(t,x),$  $p(\vr)=\Pi(\rho),$ $G=\Gamma$ 
and  $\kappa=-1$).  
By analogy with the original systems,   the case $\kappa<0$ will be named
attractive, the case $\kappa>0,$ repulsive and the case $\kappa=0,$   
pure Euler system.  For $\kappa\not=0,$ the Poisson and Helmholtz couplings
 correspond to $\mu=0$ and $\mu>0,$ respectively.  
\vskip0.25cm

Since our functional framework will force the density to tend to $0$ at infinity, the standard  
symmetrization for the compressible Euler equations is not appropriate. 
For that reason, we shall rather use the one that has been introduced by T. Makino in \cite{Ma}, 
namely,  we set 
\begin{equation}\label{eq:mak}
\rho:=\frac{2\sqrt{A\gamma}}{\gamma-1}\,\varrho^{\frac{\gamma-1}2}.
\end{equation}
After that change of unknown, System \eqref{i1}-\eqref{i0} rewrites
$$
\left\{\begin{array}{l} 
(\d_t+u\cdot\nabla)\rho+\frac{\gamma-1}2\rho\,\div u=0,\\[1.5ex]
(\d_t+u\cdot\nabla)u+\frac{\gamma-1}2\rho\,\nabla\rho=\kappa\nabla\phi,\\[1.5ex]
\Delta\phi-\mu^2\phi=\wt G\rho^{\frac2{\gamma-1}}\quad\with \wt G:= 4\pi G \Bigl(\frac{(\gamma-1)^2}{4A\gamma}\Bigr)^{\frac1{\gamma-1}}\cdotp
\end{array}\right.\leqno(MP)$$
Consider the auxiliary Cauchy problem for the $d$-dimensional
 Burgers equation:
\begin{equation}
\label{auxiso}
\partial_t  v +  v \cdot \Grad  v = 0,
\end{equation}
with initial data 
\bFormula{auxiso0}
 v(0,x)=v_0(x).
\eF
In the particular case of the 
Euler equation (that is $\kappa=0$) and under suitable spectral conditions on $Du_0,$
it is known \cite{G},\cite{GS} that \eqref{auxiso} is a good approximation  of  \eqref{i1}-\eqref{i3}
provided the density of the fluid is small enough.
The formal heuristics is that if one divides  by $\vr$, then equation \eqref{i2} becomes
\bFormula{i2bis}
\partial_t  u + u \cdot\Grad u + \frac{A\gamma}{\gamma-1}\Grad (\vr^{\gamma-1})=\kappa\Grad \phi,
\eF
and neglecting $\phi$ and $\Grad (\vr^{\gamma-1})$ in \eqref{i2bis}, we get \eqref{auxiso}.
\medbreak
In order to state our results, we need to introduce the following function  space: 
$$E^s:=\bigl\{z\in\cC(\R^d;\R^d),\; Dz\in L^\infty \andf D^2z\in H^{s-2}\bigr\}\cdotp$$
The following result  has  been first proved in \cite{G}, \cite{GS} in the case of
integer regularity exponents. Here, we extend it to  \emph{real} exponents.  
\begin{prop}
\label{p:Burgers}
Let $v_0$ be in $E^s(\R^d)$ with  $s>1+d/2$ and satisfy:
$$
\hbox{there exists }\ \varepsilon>0\ \hbox{ such that for any  }\  x\in\R^d,\ \hbox{\rm dist}(\hbox{\rm Sp}\,(Dv_0(x)),\R_-)\geq \varepsilon,
 \leqno\mathbf{(H0)}
 $$
   where  ${\rm Sp}\,A$ denotes the spectrum of the matrix $A.$ 
\medbreak
Then 
\eqref{auxiso}-\eqref{auxiso0} has a classical solution $v$ on $\R_+\times\R^d$ such that
\[
  D^2v \in \cC^j\bigl(\R_+;H^{s-2-j}(\R^d)\bigr)\ \ \ \mbox{for}\ \ j=0,1.
\]
Moreover,  $D  v\in \cC_b(\R_+\times\R^d)$ and we have for any $t\geq 0$ and  $x\in \R^d,$
\begin{equation}
\label{Du}
D  v(t,x)=(1+t)^{-1}I+(1+t)^{-2}K(t,x)
\end{equation}
for some function $K\in\cC_b(\R_+\times\R^d;\R^d\times\R^d)$ that satisfies also
\begin{equation}
\|K(t)\|_{\dot H^{\sigma}}\leq K_\sigma(1+t)^{d/2-\sigma}\ \ \mbox{for all }\ 0<\sigma\leq s-1.
\label{estim2}
\end{equation}
Finally, if $D^2v_0$ is bounded, then we have
\begin{equation}
{\displaystyle \|D^2  v(t)\|_{L^\infty}\leq C(1+t)^{-3}.}
\label{estim3}
\end{equation}
\end{prop}
The main goal of the paper is to prove the following global existence and uniqueness result
for System~\eqref{i1}-\eqref{i0}. 
\begin{thm}
\label{isentro}
Let $s>1+d/2$ and $\gamma>1.$  Assume that the initial data $(\rho_0,u_0)$ satisfy:
\begin{itemize}
\item $\mathbf{(H1)}$ there exists $v_0$ in $E^{s+1}$ satisfying  $\mathbf{(H0)}$
and such that $u_0-v_0$ is small in $H^s$;
\item $\mathbf{(H2)}$ $\vr_0^{\frac{\gamma-1}{2}}$ is small enough in $H^s.$
\end{itemize}
Denote by  $ v$  the global solution of \eqref{auxiso}-\eqref{auxiso0} given by Proposition
\ref{p:Burgers}. 
\medbreak
Then, there exists a unique global solution $\left(\vr,u,\phi\right)$ to
\eqref{i1}-\eqref{i0}, such that 
\[
\bigl(\vr^{\frac{\gamma-1}{2}},u- v\bigr) \in
\cC\bigl(\R_+;H^{s}\bigl(\R^d\bigr) \bigr),
\]
provided $d,$ $\gamma$ and $s$ satisfy the following additional conditions:
\begin{itemize}
\smallbreak\item Pure Euler case $\kappa=\mu=0$:   no additional condition on $d,$ $\gamma$ 
and $s.$
\smallbreak\item Poisson case $\kappa\not=0$ and $\mu=0$: $d\geq3,$ 
$\gamma<\min(\frac53,1+\frac4{d-1})$ and $s<\frac32+\frac{2}{\gamma-1}$
(no upper bound for $s$ if $\gamma=1+\frac2k$ for some integer $k$).
\smallbreak\item Helmholtz case $\kappa\not=0$ and $\mu\not=0$: either $(d,s,\gamma)$ are
as for the Poisson case, or 
$d\geq2,$ $s<\frac12+\frac2{\gamma-1}$ and $\gamma<1+\frac4{d+1}$
 (no upper bound for $s$  if $\gamma=1+\frac2k$ for some integer $k$).
\end{itemize}
\end{thm}
\begin{rem}\label{r:finitemass} The fact that $\varrho^{\frac{\gamma-1}2}\in \cC(\R_+;H^s(\R^d))$
implies that $\varrho\in\cC(\R_+;L^1(\R^d))$ whenever $\gamma\leq2$
(a condition which is always satisfyed in the Poisson and Helmholtz cases with $d=3$). 
Hence, the constructed solutions have finite mass, and one can show that it is conserved
through the evolution.
\end{rem}
\begin{rem} Even for the pure Euler case, our results
extend those of Grassin and Serre \cite{G2,GS,S} since we are able to treat any dimension $d\geq1$ 
and real Sobolev exponent $s>1+d/2.$ 
\end{rem}

The following decay estimates  are satisfied by  the global solutions constructed in the previous theorem (to be 
compared with those for the Burgers equation). 
\begin{thm}\label{decay}
Let all the assumptions of  Theorem \ref{isentro} be in force.  Then, for all $\sigma$ in $[0,s],$
 the solution $(\vr,u,\phi)$ constructed therein satisfies
$$
\left\|\vr^{\frac{\gamma-1}{2}},u- v\right\|_{\dot H^\sigma}\leq C_\sigma (1+t)^{\frac d2-\sigma
-\min(1,d(\frac{\gamma-1}2))},
$$
where $C_\sigma$ depends only on the initial data, on $d,$ $\gamma,$
$\mu,$  and ond $\sigma.$ 
\end{thm}
  Regarding  the Poisson coupling  in the repulsive case (that is $\mu=0$ and $\kappa>0$),  our results 
 have  to be compared with the remarkable work of  Y. Guo in \cite{G}. 
 There,  having $\kappa>0$  and  $\nabla\times u_0=0$
 is fundamental, and the result holds for small perturbations of 
 the  state $(\rho,u)=(\overline\rho,0)$ with $\overline\rho$ a \emph{positive} constant. 
 He also proved that $\|\rho(t)-\overline \rho\|_{L^\infty}$ and $\|u(t)\|_{L^\infty}$ decay as $(1+t)^{-p}$ for
 any $p<3/2$ when $t\to\infty,$ which does not correspond to the decay  we  here  establish here.

 Guo's approach is  somehow orthogonal to ours
since it strongly relies on the stability properties of the linearized Euler-Poisson 
system about $(\overline\rho,0).$


\section{Decay estimates in Sobolev spaces}\label{aux}

The goal of the present section is to prove  a priori 
decay estimates in Sobolev spaces first for the multi-dimensional 
Burgers equation \eqref{auxiso} and, next, for the discrepancy between  
the solution to  (MP)  and to \eqref{auxiso}. Those estimates  will play a fundamental role 
in the proof of our global existence result.

\subsection{Decay estimates for the Burgers equation} 

The purpose of this part is to prove Proposition \ref{p:Burgers} 
for  any  real regularity  exponent $s>1+d/2.$ 
\smallbreak
Let $X$ be the flow of $v.$ 
 The proof relies on the fact that the matrix valued function $A:(t,y)\mapsto Dv(t,X(t,y))$ 
  satisfies the Ricatti equation
  $$A'+ A^2=0,\qquad A|_{t=0}=Dv_0.$$
{}From Hypothesis $\mathbf{(H0)},$
one can deduce that $v(t,y)$ is defined for all $t\geq0$ and $y\in\R^d,$ and that 
$$Dv(t,X(t,y))= (\Id+t Dv_0(y))^{-1} Dv_0(y)\with  X(t,y)=y+tv_0(y).$$
Therefore, denoting $X_t:y\mapsto X(t,y),$ we  have
\begin{multline}\label{eq:Dv} Dv(t,x)=\frac{\Id}{1+t}+\frac{K(t,x)}{(1+t)^2}
\\ \with K(t,x):=(1+t)(\Id+tDv_0(X_t^{-1}(x)))^{-1}(Dv_0(X_t^{-1}(x))-\Id).\end{multline}
In particular, we have 
\begin{equation}\label{eq:divv}
\div v(t,y+tv_0(y))=\frac d{1+t}+\frac{{\rm Tr}\,K(t,y+tv_0(y))}{(1+t)^2}\cdotp
\end{equation}
Furthermore,   Hypothesis $\mathbf{(H0)}$ implies that
\begin{equation}\label{eq:DXbounded}
\|(\Id + tDv_0)^{-1}\|_{L^\infty}\lesssim (1+\ep t)^{-1},\end{equation}
and  $K$ is thus bounded on $\R_+\times\R^d.$
\medbreak
For the proof of \eqref{estim3} we refer to \cite{G2}. We proceed with the proof of \eqref{estim2} in the case $\sigma\in]0,1[$
(for the integer case, see \cite{G2}).  
To bound $\wt K_t:=(1+t)^{-1} K(t,\cdot)$ in $\dot H^\sigma,$ 
 we use the following characterization of Sobolev norms by finite differences:
 $$
 \|\wt K_t\|_{\dot H^\sigma}^2 \approx \int_{\R^d}\!\!\int_{\R^d}
 \frac{|\wt K_t(y)-\wt K_t(x)|^2}{{|y-x|}^{d+2\sigma}}\,dx\,dy
 $$
 and the fact that $ \wt K_t(y)-\wt K_t(x)= I^1_t(x,y)+I^2_t(x,y)$ with 

 $$\begin{aligned}
I_t^1(x,y)&= (\Id\!+\!tDv_0(X_t^{-1}(y)))^{-1}\bigl(Dv_0(X_t^{-1}(y))-Dv_0(X_t^{-1}(x))\bigr),\\
I_t^2(x,y)&=t(\Id\!+\!tDv_0(X_t^{-1}(y)))^{-1}\!\bigl(Dv_0(X_t^{-1}(x))-
\!Dv_0(X_t^{-1}(y))\bigr) \\ &\qquad \times (\Id\!+\!tDv_0(X_t^{-1}(x)))^{-1} (Dv_0(X_t^{-1}(x))-\Id).\end{aligned}$$ 
We see, thanks to \eqref{eq:DXbounded} and to the
change of variable $x'=X_t^{-1}(x)$ and $y'=X_t^{-1}(y),$ that  
$$
\int_{\R^d}\!\!\int_{\R^d}
 \frac{|I_t^1(x,y)|^2}{{|y-x|}^{d+2\sigma}}\,dx\,dy\leq\frac C{(1+\ep t)^2}
 \int_{\R^d}\!\!\int_{\R^d}\frac{|Dv_0(y')-Dv_0(x')|^2}{|X_t(y')-X_t(x')|^{d+2\sigma}}
 J_{X_t}(x') J_{X_t}(y') \,dx'\,dy'.
 $$
Therefore, using the fact that $\|J_{X_t}\|_{L^\infty}\leq C(1+\ep t)^d$ and that
$$
\begin{aligned}
|y'-x'|=|X_t^{-1}(X_t(y'))-X_t^{-1}(X_t(x'))|&\leq \|DX_t^{-1}\|_{L^\infty} |X_t(y')-X_t(x')|\\
&\leq \frac C{1+\ep t}  |X_t(y')-X_t(x')|,
\end{aligned}
$$
we get that 
\begin{equation}\label{eq:It1}
\int_{\R^d}\!\!\int_{\R^d}
 \frac{|I_t^1(x,y)|^2}{{|y-x|}^{d+2\sigma}}\,dx\,dy\leq  C(1+\ep t)^{d-2-2\sigma}
 \int_{\R^d}\!\!\int_{\R^d}
 \frac{|Dv_0(y')-Dv_0(x')|^2}{{|y'-x'|}^{d+2\sigma}}\,dx\,dy.\end{equation}
 Similarly, \eqref{eq:DXbounded} and the change of variable $x'=X_t^{-1}(x)$ and $y'=X_t^{-1}(y)$
 imply that 
 $$
\int_{\R^d}\!\!\int_{\R^d}
 \frac{|I_t^2(x,y)|^2}{{|y-x|}^{d+2\sigma}}\,dx\,dy\leq\frac{Ct^2}{(1+\ep t)^4}
  \int_{\R^d}\!\!\int_{\R^d}\frac{|Dv_0(y')-Dv_0(x')|^2}{|X_t(y')-X_t(x')|^{d+2\sigma}}
 J_{X_t}(x') J_{X_t(y')} \,dx'\,dy',
 $$ and we thus also have \eqref{eq:It1} for $I_t^2.$  As a conclusion, using
 the characterization of $\|Dv_0\|_{\dot H^\sigma}$ by finite difference, we get
 $$\|\wt K_t\|_{\dot H^\sigma}\leq C(1+\ep t)^{\frac d2-1-\sigma} \|Dv_0\|_{\dot H^\sigma},
 $$
 which gives the desired estimate for $\sigma\in]0,1[.$ 
 \medbreak
 Proving the result for higher order regularity exponents may be done by 
 taking advantage of the explicit formula for partial derivatives of $\wt K_t$
  that has been derived by M.~Grassin
 in \cite[p. 1404]{G2}. The same method as in the case $\sigma\in]0,1[$ 
 has to be applied to each term of the formula. The details are left to the reader. 
 \qed


\subsection{Sobolev estimates for System~(MP)}\label{ss:sob}

Let  $v$ be the solution of the Burgers equation given by Proposition \ref{p:Burgers}. 
Let us set   $w:=u-v$ where $(\rho,u,\phi)$ stands for a sufficiently smooth solution of $(MP)$ on $[0,T]\times\R^d.$ 
Then $(\rho,w,\phi)$ satisfies:
$$
\left\{\begin{array}{l} 
(\d_t+w\cdot\nabla)\rho+\frac{\gamma-1}2\rho\,\div w+v\cdot\nabla\rho+\frac{\gamma-1}2\rho\,\div v=0,\\[1.5ex]
(\d_t+w\cdot\nabla)w+\frac{\gamma-1}2\rho\,\nabla\rho+v\cdot\nabla w+w\cdot\nabla v=\kappa\nabla\phi,\\[1.5ex]
\Delta\phi-\mu^2\phi=\wt G\rho^{\frac2{\gamma-1}}.\end{array}\right.\leqno(BB)
$$
Our aim is to prove decay estimates in $\dot H^\sigma$ for $(BB),$ for all $0\leq\sigma\leq s.$
Clearly, arguing by interpolation, it suffices to consider the border cases $\sigma=0$ and $\sigma=s.$
\medbreak
Let us start with  $\sigma=0.$  Taking the $L^2$ 
scalar product of the first two equations of  $(BB)$ with $(\rho,w)$ gives
\begin{multline}\label{eq:L20}
\frac12\frac d{dt}\|(\rho,w)\|_{L^2}^2-\frac12\int_{\R^d}(\rho^2+|w|^2)\div v\,dx
-\frac12\int_{\R^d}(\rho^2+|w|^2)\div w\,dx\\+\frac{\gamma-1}2\int_{\R^d}\rho^2\div v\,dx
+\int_{\R^d}(w\cdot\nabla v)\cdot w\,dx +\frac{\gamma-1}4 \int_{\R^d} \rho^2 \div w\, dx=\kappa\int_{\R^d}\nabla\phi\cdot w\,dx.
\end{multline}
Let 
\begin{equation}\label{eq:cdg} c_{d,\gamma}:= \min\left(1,d\frac{\gamma-1}2\right)-\frac d2\cdotp\end{equation}
 {}From \eqref{eq:L20}, Cauchy-Schwarz inequality and \eqref{eq:divv}, we deduce that, denoting by $M$ a bound of $K,$ 
\begin{multline}\label{eq:L2}
\frac12\frac d{dt}\|(\rho,w)\|_{L^2}^2+\frac{c_{d,\gamma}}{1+t}\|(\rho,w)\|_{L^2}^2\leq
|\kappa|\|\nabla\phi\|_{L^2}\|w\|_{L^2}\\+\frac{M\max(1+\frac d2,|\gamma-2|\frac d2)}{(1+t)^2}\|(\rho,w)\|_{L^2}^2+\max\left(\frac12,\frac{|\gamma-3|}4\right)\|\div w\|_{L^\infty}\|(\rho,w)\|_{L^2}^2.
\end{multline}
Even if  $\kappa=0$ (i.e. standard compressible Euler equation), proving existence results for $(BB)$ 
requires a control on $\|\div w\|_{L^\infty}.$ Owing to the hyperbolicity of the system, 
it seems (at least in the multi-dimensional case) difficult to go beyond the energy framework, 
and it is thus natural to look for a priori estimates in  Sobolev spaces $H^s.$
Now, owing to Sobolev embedding, the minimal requirement to get eventually a bound on 
 $\|\div w\|_{L^\infty}$ is that $s>1+d/2.$ 
  
In order to prove Sobolev estimates, we introduce  the homogeneous fractional derivation operator 
 $\dot\Lambda^s$  defined by  $\cF(\dot\Lambda^sf)(\xi):=|\xi|^s\cF f(\xi)$ and observe that  
 $\rho_s:=\dot\Lambda^s\rho,$ $w_s:=\dot\Lambda^sw$ and $\phi_s:=\dot\Lambda^s\phi$ satisfy
(with the usual summation convention over repeated indices)
$$
\left\{\begin{array}{l} 
(\d_t+w\cdot\nabla)\rho_s+\frac{\gamma-1}2\rho\,\div w_s+v\cdot\nabla\rho_s-s\d_jv^k\dot\Lambda^{-2}\d^2_{jk}\rho_s
+\frac{\gamma-1}2\dot\Lambda^s(\rho\,\div v)\\ 
\hspace{9.5cm}=\dot R^1_s+\dot R^2_s+\dot R^3_s,\\[1.5ex]
(\d_t+w\cdot\nabla)w_s+\frac{\gamma-1}2\rho\nabla\rho_s+v\cdot\nabla w_s-s\d_jv^k\dot\Lambda^{-2}\d^2_{jk}w_s+\dot\Lambda^s(w\cdot\nabla v)\\
\hspace{8cm}=\dot R^4_s+\dot R^5_s+\dot R^6_s+\kappa\nabla\phi_s,\\[1.5ex]
\Delta\phi_s-m^2\phi_s=\wt G\dot\Lambda^s\bigl(\rho^{\frac2{\gamma-1}}\bigr),\end{array}\right.\leqno(BB_s)
$$
with 
$$\begin{array}{ll}
\dot R^1_s:=[w,\dot\Lambda^s]\nabla\rho,\quad & \dot R^4_s:=[w,\dot\Lambda^s]\nabla w,\\[1ex]
\dot R^2_s:= \frac{\gamma-1}2[\rho,\dot\Lambda^s]\div w,\quad&\dot R^5_s:=\frac{\gamma-1}2[\rho,\dot\Lambda^s]\nabla\rho,\\[1ex]
\dot R^3_s:=[v,\dot\Lambda^s]\nabla\rho-s\d_jv^k\dot\Lambda^{-2}\d^2_{jk}\rho_s,\quad&
\dot R^6_s:=[v,\dot\Lambda^s]\nabla w-s\d_jv^k\dot\Lambda^{-2}\d^2_{jk}w_s.
\end{array}
$$
The definition of $\dot R^3_s$ and $\dot R^6_s$ is motivated
by the fact that, according to  the classical  theory of pseudo-differential operators, we  expect to have 
$$
[\dot \Lambda^s,v]\cdot\nabla z=\frac1i\bigl\{|\xi|^s,v(x)\bigr\}(D)\nabla z +\hbox{remainder}.
$$
Computing  the \emph{Poisson bracket}  in the right-hand side  yields  
 $$
 \frac1i\bigl\{|\xi|^s,v(x)\bigr\}(D) = -s \d_j v \dot\Lambda^{s-2} \d_j.
 $$
 Now, taking advantage of \eqref{eq:Dv}, we get 
$$-\d_jv^k\dot\Lambda^{-2}\d^2_{jk}z=\frac1{1+t}z-\frac{K_{kj}}{(1+t)^2}\dot\Lambda^{-2}\d^2_{jk}z,
$$
and using \eqref{eq:divv} yields
$$\begin{aligned}
\dot\Lambda^s(\rho\,\div v)&=\frac d{1+t}\rho_s+\frac1{(1+t)^2}\dot\Lambda^s(\rho\,{\rm Tr}\, K)\\
\andf \dot \Lambda^s(w\cdot\nabla v)&=\frac1{1+t}w_s+\frac1{(1+t)^2}\dot\Lambda^s(K\cdot w).\end{aligned}$$
Hence, taking the $L^2$ inner product of $(BB_s)$ with $(\rho_s,w_s),$  and denoting 
\begin{equation}\label{eq:cdgs}c_{d,\gamma,s}:= c_{d,\gamma}+s,\end{equation}
 we discover that 
\begin{multline}\label{eq:Hs}
\frac12\frac d{dt}\|(\rho_s,w_s)\|_{L^2}^2+\frac{c_{d,\gamma,s}}{1+t}\|(\rho_s,w_s)\|_{L^2}^2\leq
|\kappa|\|\nabla\phi_s\|_{L^2}\|w_s\|_{L^2}\!\\+\!\frac{\|\div
  w\|_{L^\infty}}2\|(\rho_s,w_s)\|_{L^2}^2+\frac{\gamma-1}2\|\nabla \rho\|_{L^\infty}\|\rho_s\|_{L^2}\|w_s\|_{L^2}
+\frac{sM}{(1+t)^2} \|(\rho_s,w_s)\|_{L^2}^2\\+\biggl(\frac{1}{(1\!+\!t)^2}\left(\frac{\gamma-1}2\|\dot\Lambda^s(\rho{\rm Tr}\, K)\|_{L^2}+\|\dot\Lambda^s(K\cdot w)\|_{L^2}\right) 
+\sum_{j=1}^6\|\dot R_s^j\|_{L^2}\biggr)\|(\rho_s,w_s)\|_{L^2}.
\end{multline}
The  terms $\dot R_s^1,$ $\dot R_s^2,$ $\dot R_s^4$ and $\dot R_s^5$  may be treated according to the homogeneous  version  of Kato and Ponce commutator estimates (see Lemma \ref{l:com1} in  the Appendix). 
We get
$$\begin{aligned}
\|\dot R^1_s\|_{L^2}&\lesssim \|\nabla\rho\|_{L^\infty}\|\nabla w\|_{\dot H^{s-1}}+\|\nabla w\|_{L^\infty}\|\rho\|_{\dot H^{s}},\\
\|\dot R_s^4\|_{L^2}&\lesssim \|\nabla w\|_{L^\infty}\|w\|_{\dot H^{s}},\\
\|\dot R_s^2\|_{L^2}&\lesssim \|\div w\|_{L^\infty}\|\rho\|_{\dot H^{s}}+\|\nabla \rho\|_{L^\infty}\|\div w\|_{\dot H^{s-1}},\\
\|\dot R_s^5\|_{L^2}&\lesssim \|\nabla\rho\|_{L^\infty}\|\rho\|_{\dot H^{s}}.\end{aligned}$$
The (more involved) terms $\dot R_s^3$ and $\dot R_s^6$  may be handled thanks to  Lemma~\ref{l:com2}.
We get
$$\begin{aligned}\|\dot R_s^3\|_{L^2}
&\lesssim \|\nabla\rho\|_{L^\infty}\|v\|_{\dot H^{s}}+\|\nabla^2v\|_{L^\infty}\|\nabla\rho\|_{\dot H^{s-2}},\\
\|\dot R_s^6\|_{L^2}&\lesssim \|\nabla w\|_{L^\infty}\|v\|_{\dot H^{s}}+\|\nabla^2v\|_{L^\infty}\|\nabla w\|_{\dot H^{s-2}}.\end{aligned}$$
Finally, the standard Sobolev tame estimate yields 
$$\begin{aligned}
\|\dot\Lambda^s(\rho{\rm Tr}\, K)\|_{L^2}
&\lesssim\|\rho\|_{L^\infty}\|\nabla({\rm Tr}\, K)\|_{\dot H^{s-1}}+\|{\rm Tr}\, K\|_{L^\infty}\|\rho\|_{\dot H^s}\\
\andf \|\dot\Lambda^s(K\cdot w)\|_{L^2}&\lesssim\|w\|_{L^\infty}\|\nabla K\|_{\dot H^{s-1}}+\|K\|_{L^\infty}\|w\|_{\dot H^s}.\end{aligned}
$$
Plugging all the above estimates in \eqref{eq:Hs} and using Proposition \ref{p:Burgers},   we end up with 
\begin{multline}\label{eq:Hs2}
\frac12\frac d{dt}\|(\rho,w)\|_{\dot H^s}^2+\frac{c_{d,\gamma,s}}{1+t}\|(\rho,w)\|_{\dot H^s}^2\leq
|\kappa|\|\nabla\phi\|_{\dot H^s}\|w\|_{\dot H^s}\\+\frac{CM}{(1+t)^2}\|(\rho,w)\|_{\dot H^s}^2
+C\|D^2v\|_{\dot H^{s-1}}\|(\rho,w)\|_{L^\infty}\|(\rho,w)\|_{\dot H^s}
\\+\|(\nabla\rho,\nabla w)\|_{L^\infty}\|(\rho,w)\|_{\dot H^s}^2+\|\nabla^2v\|_{L^\infty}\|(\rho,w)\|_{\dot H^{s-1}}
\|(\rho,w)\|_{\dot H^s}.\end{multline}
Let us introduce the  notation
$$\dot X_\sigma:=\|(\rho,w)\|_{\dot H^\sigma}\andf
X_\sigma:=\sqrt{\dot X_0^2+\dot X_\sigma^2}\approx \|(\rho,w)\|_{H^\sigma}
 \quad\hbox{for }\ \sigma\geq0.$$
Our aim is to bound the right-hand side of \eqref{eq:L2} and  \eqref{eq:Hs2} in terms of $\dot X_0$ and $\dot X_s$ only. 

\subsubsection*{The standard compressible Euler equations ($\kappa=\mu=0$)}
Arguing by interpolation, we get 
\begin{eqnarray}\label{eq:interpo1}
\|(\rho,w)\|_{L^\infty}&\!\!\!\lesssim\!\!\!& \dot X_0^{1-\frac d{2s}}\dot X_s^{\frac d{2s}},\\\label{eq:interpo2}
\|(D\rho,Dw)\|_{L^\infty}&\!\!\!\lesssim\!\!\!& \dot X_0^{1-\frac1s(\frac d2+1)}\dot X_s^{\frac 1s(\frac d2+1)},\\\label{eq:interpo3}
\|(\rho,w)\|_{\dot H^{s-1}}&\!\!\!\lesssim\!\!\!& \dot X_0^{\frac1s}\dot X_s^{1-\frac1s}.
\end{eqnarray}
Then, plugging these inequalities and those of Proposition \ref{p:Burgers} in \eqref{eq:L2} and \eqref{eq:Hs2} yields
$$
\begin{aligned}
\frac d{dt}\dot X_0+\frac{c_{d,\gamma}}{1+t}\dot X_0&\lesssim \frac{\dot X_0}{(1+t)^2}+\dot X_0^{2-\frac1s(\frac d2+1)}\dot X_s^{\frac1s(\frac d2+1)},\\[1ex]
\frac d{dt}\dot X_s+\frac{c_{d,\gamma,s}}{1+t}\dot X_s&\lesssim \frac{\dot X_s}{(1+t)^2}+\frac{\dot X_0^{1-\frac d{2s}}\dot X_s^{\frac d{2s}}}{(1+t)^{s+2-\frac d2}}
+\dot X_0^{1-\frac 1s(\frac d2+1)}\dot X_s^{1+\frac1s(\frac d2+1)}+\frac{\dot X_0^{\frac1s}\dot X_s^{1-\frac1s}}{(1+t)^3}\cdotp
\end{aligned}
$$

Since we expect to have  $\dot X_\sigma\lesssim (1+t)^{-c_{d,\gamma,\sigma}}$ for all $\sigma\in[0,s],$
it is natural  to introduce the function $\dot Y_\sigma:=(1+t)^{c_{d,\gamma,\sigma}}\dot X_\sigma.$
However, for technical reasons, we proceed as in \cite{G2} and  work with 
\begin{equation}\label{eq:Y}
\dot Y_\sigma:=(1+t)^{c_{d,\gamma,\sigma}-a}\dot X_\sigma\quad\hbox{for some }\ a>1.\end{equation}
Then, observing that
$$
\frac d{dt}\dot Y_\sigma +\frac a{1+t}\dot Y_\sigma
=(1+t)^{c_{d,\gamma,\sigma}-a}\biggl(\frac d{dt}\dot X_\sigma+\frac{c_{d,\gamma,\sigma}}{1+t}\dot X_\sigma\biggr),
$$
the above inequalities for $\dot X_0$ and $\dot X_s$ lead us to
$$
\begin{aligned}
\frac d{dt}\dot Y_0+\frac{a}{1+t}\dot Y_0&\lesssim \frac{\dot Y_0}{(1+t)^2}
+\frac{\dot Y_0^{2-\frac1s(\frac d2+1)}\dot Y_s^{\frac1s(\frac d2+1)}}{(1+t)^{1+\frac d2+c_{d,\gamma}-a}},\\[1ex]
\frac d{dt}\dot Y_s+\frac{a}{1+t}\dot Y_s&\lesssim \frac{\dot Y_s}{(1+t)^2}+\frac{\dot Y_0^{1-\frac d{2s}}\dot Y_s^{\frac d{2s}}}{(1+t)^{2}}
+\frac{\dot Y_0^{1-\frac 1s(\frac d2+1)}\dot Y_s^{1+\frac1s(\frac d2+1)}}{(1+t)^{1+\frac d2+c_{d,\gamma}-a}}
+\frac{\dot Y_0^{\frac1s}\dot Y_s^{1-\frac1s}}{(1+t)^2},
\end{aligned}
$$
whence, introducing the notation $Y_\sigma:=\sqrt{\dot Y_0^2+\dot Y_\sigma^2},$
$$
\frac d{dt}Y_s+\frac a{1+t}Y_s\lesssim\frac{Y_s}{(1+t)^2}+\frac{Y_s^2}{(1+t)^{1+\frac d2+c_{d,\gamma}-a}}\cdotp
$$
At this stage, one can take $a=1+\frac d2+c_{d,\gamma}$ so that $a>1$ is satisfied as soon as $\gamma>1.$
Then, we eventually get the differential inequality 
$$
\frac d{dt}Y_s+\frac a{1+t}Y_s\lesssim\frac{Y_s}{(1+t)^2}+Y_s^2,
$$
from which one can get a global control of $Y_s$ \emph{without any restriction on 
$\gamma>1$ and $s>\frac d2+1$}, see Proposition 4 in \cite{G2}, or 
Lemma \ref{l:ODE} below\footnote{Having $\gamma=1$ would 
force us to take $a=1$ in Lemma \ref{l:ODE}, in which case the expected inequality 
unfortunately fails.}. 

\subsubsection*{Poisson coupling ($\kappa\not=0$ and $\mu=0$)}
Compared to the previous paragraph, one  has to bound the additional term $\nabla\phi$ in $H^s$ or, 
equivalently, 
 $\nabla(-\Delta)^{-1}(\rho^{\frac{2}{\gamma-1}})$ in $L^2$ and in $\dot H^s.$

To bound the $L^2$ norm,  we will have to assume that $d\geq3.$ Then, by virtue of the Sobolev embedding  
$L^p\hookrightarrow \dot H^{-1}$ with $\frac dp=1+\frac d2,$ we get
$$
  \|\nabla(-\Delta)^{-1}(\rho^{\frac{2}{\gamma-1}})\|_{L^2}\lesssim
  \|\rho^{\frac{2}{\gamma-1}}\|_{L^p}\lesssim \|\rho\|_{L^q}^{\frac2{\gamma-1}}\with
  q=\frac{2p}{\gamma-1}=\frac{4d}{(\gamma-1)(d+2)}\cdotp$$ 
   Bounding the $L^q$ norm  from the Sobolev norm $H^s$ requires $q\geq2,$ whence the constraint
\begin{equation}\label{eq:g1}1<\gamma\leq\frac{3d+2}{d+2}\cdotp\end{equation}
Then, remembering the definition of $q,$ one may argue by 
interpolation as follows:
$$
  \|\nabla(-\Delta)^{-1}(\rho^{\frac{2}{\gamma-1}})\|_{L^2}\lesssim
  \Bigl(\|\rho\|_{L^2}^{1-\theta}\|\rho\|_{\dot H^s}^\theta\Bigr)^{\frac2{\gamma-1}}
  \with \theta=\frac1s\biggl(\frac d2-\frac{(\gamma-1)(d+2)}4\biggr)\cdotp
  $$
Using the definition of $\theta,$ this eventually leads to 
  $$
\frac d{dt}\dot Y_0+\frac{a}{1+t}\dot Y_0\lesssim \frac{\dot Y_0}{(1+t)^2}
+\frac{\dot Y_0^{2-\frac1s(\frac d2+1)}\dot Y_s^{\frac1s(\frac d2+1)}}{(1+t)^{1+\frac d2+c_{d,\gamma}-a}}
+\frac{(\dot Y_0^{1-\theta}\dot Y_s^\theta)^{\frac2{\gamma-1}}}{(1+t)^{(c_{d,\gamma}+\frac d2-a)(\frac2{\gamma-1}-1)-1}}\cdotp
$$
Bounding the  $\dot H^s$ norm of the potential term relies on   Lemma \ref{l:compo} in the Appendix, 
from which we get if $\frac2{\gamma-1}\geq1,$ taking $z=\rho,$ $\alpha=\frac2{\gamma-1}$ and $\sigma=s-1,$  
\begin{equation}\label{eq:gammas}\|\nabla(-\Delta)^{-1}\rho^{\frac2{\gamma-1}}\|_{\dot H^s}\leq\|\rho^{\frac2{\gamma-1}}\|_{\dot H^{s-1}}\lesssim \|\rho\|_{L^\infty}^{\frac2{\gamma-1}-1}\|\rho\|_{\dot H^{s-1}}
 \quad\hbox{if}\quad 0\leq s-1<\frac2{\gamma-1}+\frac12\cdotp\end{equation}
 Since we need   $s>1+\frac d2,$ Inequality \eqref{eq:gammas} yields the 
following additional  restrictions on~$\gamma$:
\begin{equation}\label{eq:gammaH}
1<\gamma\leq 3\andf \gamma <1+\frac 4{d-1}\cdotp
\end{equation}
{}From that point, arguing as in the previous paragraph, we obtain that
$$\displaylines{
\frac d{dt}\dot Y_s+\frac{a}{1+t}\dot Y_s\lesssim \frac{\dot Y_s}{(1+t)^2}+\frac{\dot Y_0^{1-\frac d{2s}}\dot Y_s^{\frac d{2s}}}{(1+t)^{2}}
+\frac{\dot Y_0^{1-\frac 1s(\frac d2+1)}\dot Y_s^{1+\frac1s(\frac d2+1)}}{(1+t)^{1+\frac d2+c_{d,\gamma}-a}}
\hfill\cr\hfill+\frac{\dot Y_0^{\frac1s}\dot Y_s^{1-\frac1s}}{(1+t)^2}
+\frac{(\dot Y_0^{1-\frac d{2s}}\dot Y_s^{\frac d{2s}})^{\frac2{\gamma-1}-1}\dot Y_0^{\frac1s}\dot Y_s^{1-\frac1s}}{(1+t)^{(c_{d,\gamma}+\frac d2-a)(\frac2{\gamma-1}-1)-1}}\cdotp}$$
      Therefore, still denoting $Y_s:=\sqrt{\dot Y_0^2+\dot Y_s^2},$  we get
\begin{equation}\label{eq:YsP}
\frac d{dt}Y_s+\frac a{1+t}Y_s\lesssim\frac{Y_s}{(1+t)^2}+\frac{Y_s^2}{(1+t)^{1+\frac d2+c_{d,\gamma}-a}}
+\frac{Y_s^{\frac2{\gamma-1}}}{(1+t)^{(c_{d,\gamma}+\frac d2-a)(\frac2{\gamma-1}-1)-1}}\cdotp
\end{equation}  
At this stage,  one may 
apply  Lemma \ref{l:ODE} with 
$$m=\frac2{\gamma-1}-1\andf m'=2+\biggl(\frac2{\gamma-1}-1\biggr)\biggl(a-c_{d,\gamma}-\frac d2\biggr)
$$
and eventually get, provided $\|(\rho_0,w_0)\|_{H^s}$ is small enough:
\begin{equation}\label{eq:poissonHs}
\sqrt{(1+t)^{2s}\|(\rho,w)\|_{\dot H^s}^2+\|(\rho,w)\|_{L^2}^2}\leq 2\frac{e^{\frac{Ct}{1+t}}}{(1+t)^{c_{d,\gamma}}} \|(\rho_0,w_0)\|_{H^s}.
\end{equation}
Let us emphasize that in order to apply Lemma \ref{l:ODE}, we need  $a>1$ and 
$m'<ma,$ the second condition being equivalent to
$$
2+\biggl(\frac2{\gamma-1}-1\biggr)\biggl(a-c_{d,\gamma}-\frac d2\biggr)<a\biggl(\frac2{\gamma-1}-1\biggr),
$$
that is to say
$$
\min\biggl(1,d\Bigl(\frac{\gamma-1}2\Bigr)\biggr)\biggl(\frac2{\gamma-1}-1\biggr)>2.
$$
For  $d\geq3,$  that latter inequality is equivalent to $\gamma<5/3$ (and for $d\leq 2,$ it is never satisfied).
Keeping in mind the constraints \eqref{eq:g1},  \eqref{eq:gammas} and  \eqref{eq:gammaH}, 
one can conclude that \eqref{eq:poissonHs} holds true whenever 
\begin{equation}\label{eq:condP}1<\gamma<\min\biggl(\frac 53,1+\frac 4{d-1}\biggr),\quad 1+\frac d2<s<\frac32+\frac2{\gamma-1}\andf d\geq3.\end{equation}
Of course, \eqref{eq:poissonHs} is valid for all $s>1+\frac d2$ in the case where  $\frac2{\gamma-1}$ is an integer, 
that is to say, if $\gamma=1+\frac2k$ for some integer $k>\max(3,\frac d2)\cdotp$

\subsubsection*{Helmholtz coupling}
Let us finally consider  the case  $\mu>0$ and $\kappa\not=0.$
Since $\mu>0,$ we have on the one hand,
if  $\frac2{\gamma-1}-1\geq0$ (that is $\gamma\leq3$), 
$$\|\nabla(\mu^2-\Delta)^{-1}\rho^{\frac2{\gamma-1}}\|_{L^2}\lesssim
 \|\nabla(\rho^{\frac{2}{\gamma-1}})\|_{L^2}\lesssim\|\rho\|_{L^\infty}^{\frac2{\gamma-1}-1}\|\nabla\rho\|_{L^2}$$
and, on the other hand, as may be seen after decomposing into small and large $\xi$'s 
in the definition  of  the norm in $\dot H^s,$ 
$$\|\nabla(\mu^2-\Delta)^{-1}\rho^{\frac2{\gamma-1}}\|_{\dot H^s}\lesssim \|\rho^{\frac2{\gamma-1}}\|_{\dot H^{s-1}}.$$
The right-hand side may be bounded according to \eqref{eq:gammas}. 
Hence, arguing as in  the case $m=0,$  we end up with
$$\begin{aligned}
\frac d{dt}\dot Y_0+\frac{a}{1+t}\dot Y_0&\lesssim \frac{\dot Y_0}{(1+t)^2}
+\frac{\dot Y_0^{2-\frac1s(\frac d2+1)}\dot Y_s^{\frac1s(\frac d2+1)}}{(1+t)^{1+\frac d2+c_{d,\gamma}-a}}
+\frac{\bigl(\dot Y_0^{1-\frac d{2s}}\dot Y_s^{\frac d{2s}}\bigr)^{\frac2{\gamma-1}-1}
\dot Y_0^{1-\frac1s}\dot Y_s^{\frac1s}}{(1+t)^{(\frac d2+c_{d,\gamma}-a)(\frac2{\gamma-1}-1)+1}},\\[1ex]
\frac d{dt}\dot Y_s+\frac{a}{1+t}\dot Y_s&\lesssim \frac{\dot Y_s}{(1+t)^2}+\frac{\dot Y_0^{1-\frac d{2s}}\dot Y_s^{\frac d{2s}}}{(1+t)^{2}}
+\frac{\dot Y_0^{1-\frac 1s(\frac d2+1)}\dot Y_s^{1+\frac1s(\frac d2+1)}}{(1+t)^{1+\frac d2+c_{d,\gamma}-a}}
+\frac{\dot Y_0^{\frac1s}\dot Y_s^{1-\frac1s}}{(1+t)^2}\\
&\hspace{5cm}+\frac{\bigl(\dot Y_0^{1-\frac d{2s}}\dot Y_s^{\frac d{2s}}\bigr)^{\frac2{\gamma-1}-1}
\dot Y_0^{\frac1s}\dot Y_s^{1-\frac1s}}{(1+t)^{(\frac d2+c_{d,\gamma}-a)(\frac2{\gamma-1}-1)-1}}\cdotp
\end{aligned}
$$
As expected, the estimate for $\dot Y_0$ is ``better'' than in the Poisson case and, at this stage, one does not
have any constraint on the dimension. Unfortunately, this is not of much help since the inequality for $\dot Y_s$ is the same as before, 
 leading us again   to 
\begin{equation}\label{eq:YsH}
\frac d{dt}Y_s+\frac a{1+t}Y_s\lesssim\frac{Y_s}{(1+t)^2}+\frac{Y_s^2}{(1+t)^{1+\frac d2+c_{d,\gamma}-a}}
+\frac{Y_s^{\frac2{\gamma-1}}}{(1+t)^{(\frac d2+c_{d,\gamma}-a)(\frac2{\gamma-1}-1)-1}}
\end{equation}
and thus to \eqref{eq:poissonHs} under condition \eqref{eq:condP}. 
\medbreak
To handle the case $d=2,$ one may use the fact that we also have
$$\|\nabla(\mu^2-\Delta)^{-1}\rho^{\frac2{\gamma-1}}\|_{\dot H^s}\lesssim \|\rho^{\frac2{\gamma-1}}\|_{\dot H^{s}}.$$ 
Hence, using Lemma \ref{l:compo} yields
$$\|\nabla(\mu^2-\Delta)^{-1}\rho^{\frac2{\gamma-1}}\|_{\dot H^s}\lesssim 
\|\rho\|_{L^\infty}^{\frac2{\gamma-1}-1} \|\rho\|_{\dot H^s}\quad\hbox{if }\ 
s<\frac2{\gamma-1}+\frac12\cdotp
$$
Consequently the denominator of the inequality for $\dot Y_s$ becomes 
$(1+t)^{(\frac d2+c_{d,\gamma}-a)(\frac2{\gamma-1}-1)}$ so that, using Lemma \ref{l:ODE},
we now arrive, if $d=2,$  at the constraints
$$ 1<\gamma <\frac73\andf  s<\frac2{\gamma-1}+\frac12\cdotp$$
Using the same method as for  $d\geq3$ (one has to take $m'=1+\bigl(\frac2{\gamma-1}\bigr)
\bigl(a-c_{d,\gamma}-\frac d2\bigr)$ and $m=\frac2{\gamma-1}-1$ in Lemma \ref{l:ODE}), 
one  can still complete  the proof if
$$
1<\gamma<1+\frac4{d+1}\andf s<\frac2{\gamma-1}+\frac12\cdotp
$$
Note that, if $\gamma=1+\frac2k\,$ for some $k>2,$ then $\frac2{\gamma-1}$ is an integer so that 
the only remaining constraint on $s$ is that  $s>1+\frac d2\cdotp$




\section{Proving  Theorem \ref{isentro}}

A number of works have been devoted to the local existence issue for the Euler-Poisson system, 
in various functional settings (see e.g.  Makino \cite{Ma}, Gamblin \cite{ga}, B\'ezard \cite{be} and Brauer-Karp \cite{BK}). 
However, to the best of our knowledge, none of them treats also the case $\mu\not=0$ and Sobolev spaces 
with fractional regularity (furthermore, our data are not exactly in uniformly local Sobolev spaces). 
For the reader's convenience, we sketch the proof of global existence
for \eqref{i1}--\eqref{i0} in the functional setting of Theorem \ref{isentro}, 
then establish  uniqueness by means of a classical energy method.

\subsection{Existence} 

Here we are given $(\rho_0,u_0)$ satisfying the assumptions of Theorem \ref{isentro}. 
Our goal is to prove the existence of a global-in-time solution. 
\subsubsection*{Step 1: Solving an approximate system}
Fix some cut-off function $\chi\in \cC^\infty_c(\R^d)$ supported in, say, the ball $B(0,4/3)$ and 
with value $1$ on $B(0,3/4).$ Set $v^n:=\chi(n^{-1}\cdot)\,v.$ 
Let $J_n$ be the Friedrichs' truncation operator 
defined by  $J_n z:=\cF^{-1}(1_{B(0,n)} \cF z).$ 

For all $n\geq1,$ we consider the following regularization  of~$(BB)$:
$$
\left\{\begin{array}{l} 
\d_t\rho+J_n((v^n\!+\!J_nw)\cdot\nabla J_n\rho)+\frac{\gamma-1}2J_n(J_n\rho\,\div(v^n\!+\!J_nw))=0,\\[1.5ex]
\d_tw+J_n((v^n\!+\!J_nw)\cdot\nabla J_nw)+\frac{\gamma-1}2J_n(J_n\rho\,\nabla J_n\rho)+J_n(J_nw\cdot\nabla v^n)\\
\hspace{6cm}=-\kappa\wt G J_n\nabla(\mu^2\Id-\Delta)^{-1}(J_n\rho)^{\frac2{\gamma-1}},\end{array}\right.\leqno(BB_n)$$
 supplemented with initial data $(J_n\rho_0,J_nw_0).$ 
\medbreak
Note that $v^n$ is in $\cC(\R_+;H^{s+1}).$ Hence the above system may be seen as an ODE in $L^2(\R^d;\R\times\R^d).$ Applying the standard Cauchy-Lipschitz theorem thus 
ensures that there exists a unique maximal solution $(\rho^n,w^n)\in\cC^1([0,T^n);L^2)$ to $(BB_n).$

Now, from  $J_n^2=J_n,$ we deduce that  $(J_n\rho^n,J_nw^n)$ also satisfies $(BB_n).$ Hence, uniqueness of the solution 
entails that $J_n\rho^n=\rho^n$ and $J_nw^n=w^n.$ In other words, $(\rho^n,w^n)$ is spectrally localized in the ball $B(0,n),$
and one can thus assert that  $(\rho^n,w^n)\in\cC^1([0,T^n);H^\sigma)$ for all $\sigma$ in $\R,$ 
and  actually satisfies: 
\begin{equation}\label{eq:BBn}
\left\{\begin{array}{l} 
\d_t\rho^n+J_n((v^n\!+\!w^n)\cdot\nabla\rho^n)+\frac{\gamma-1}2J_n(\rho^n\,\div(v^n\!+\!w^n))=0,\\[1.5ex]
\d_tw^n+J_n((v^n\!+\!w^n)\cdot\nabla w^n)+\frac{\gamma-1}2J_n(\rho^n\,\nabla\rho^n)+J_n(w^n\cdot\nabla v^n)\\
\hspace{6cm}=-\kappa\wt G J_n\nabla(\mu^2\Id-\Delta)^{-1}(\rho^n)^{\frac2{\gamma-1}},\\[1.5ex]
(\rho^n,w^n)|_{t=0}=(J_n\rho_0,J_nw_0).\end{array}\right.\end{equation}

\subsubsection*{Step 2: Uniform a priori estimates in the solution space}
Since  $J_n$ is an orthogonal projector in any Sobolev space and  $(J_n\rho^n,J_nw^n)=(\rho^n,w^n),$
 one can repeat verbatim (and rigorously) the computations of subsection \ref{ss:sob}. The only change is that 
 since $Dv^n=Dv +{\mathcal O}(n^{-1}),$  the final estimates therein
 only hold on the time interval $[0,\min(cn,T_n))$ for some $c>0,$
 which eventually implies that  
   $T_n\geq cn.$

The conclusion of this step is that for any fixed $T>0,$ the couple 
 $(\rho^n,w^n)$  for $n$ large enough is defined on $[0,T],$  
 belongs to $\cC^1([0,T];H^\sigma)$ for all $\sigma\in\R$ and is
bounded in $L^\infty([0,T];H^s).$

\subsubsection*{Step 3: Convergence}    Let us fix some $T>0.$ 
Given the uniform bounds of the previous step, the weak $*$ compactness theorem 
ensures that (up to an omitted extraction),  there exists some $(\rho,w)\in L^\infty([0,T];H^s)$ so that 
$$(\rho^n,w^n)\rightharpoonup (\rho,w) \ \hbox{ weak }\ * \ \hbox{ in }\  L^\infty([0,T];H^s).$$
Furthermore, computing $\d_t\rho^n$ and $\d_tw^n$ by means of \eqref{eq:BBn} and using 
standard product laws in Sobolev spaces, one can prove that  for all $\theta\in \cC^\infty_c(\R^d),$ 
$(\theta\d_t\rho^n,\theta\d_tw^n)$ is bounded in $L^\infty([0,T]; H^{s-1}).$ 
Hence, from  Aubin-Lions lemma, interpolation and Cantor diagonal process, we gather that (still up to an omitted extraction), 
we have 
$$(\theta\rho^n,\theta w^n)\to (\theta\rho,\theta w)\quad\hbox{in}\quad L^\infty([0,T];H^{s'})
\quad\hbox{for all }\ \theta\in \cC^\infty_c(\R^d)\andf s'<s.$$
This allows to pass to the limit in $\eqref{eq:BBn}$ and to conclude that 
$(\rho,w)$ satisfies $(BB)$ on $[0,T]\times\R^d.$ Of course, since $T>0$ is arbitrary, $(\rho,w)$
actually satisfies  $(BB)$ on $\R_+\times\R^d$ and belongs to $L^\infty_{loc}(\R_+;H^s).$ 

\subsubsection*{Step 4: Time continuity}  
That $(\rho,w)$ lies in $\cC(\R_+;H^s)$ may be achieved either by adapting the arguments of Kato in \cite{K}
or those of \cite[Chap. 4]{BCD}. 
Note that ref. \cite{K} allows in addition to prove the continuity of the flow map in the space $\cC(\R_+;H^s).$ 



\subsection{Uniqueness}

This part is devoted to the proof of the following uniqueness result.

\begin{prop}\label{locunik}
Let  the assumptions of Theorem \ref{isentro} be in force and 
assume that $(\varrho^1,u^1)$ and $(\varrho^2,u^2)$ are two solutions of $(MP)$ on $[0,T]\times\R^d$
such that $(\varrho^i)^{\frac{\gamma-1}2} \in \cC([0,T];L^2\cap L^\infty),$
$\nabla\varrho^i\in L^\infty\left([0,T]\times\R^d\right),$ 
$(u^i-v^i)\in \cC([0,T[;L^2)$ and $\nabla u^i\in L^1([0,T];L^\infty)$ for $i=1,2.$
\medbreak
If, in addition, $(\varrho^1,u^1)$ and $(\varrho^2,u^2)$ coincide at time $t=0,$
then the two solutions are the same on $[0,T]\times\R^d.$ 
\end{prop}
{\bf Proof:}  We define $w^i = u^i-v^i$ for $i=1,2$, and
$(\dr,\dw,\df):=(\rho^2-\rho^1,w^2-w^1,\phi^2-\phi^1)$. Since $(\varrho^1,u^1)$ and $(\varrho^2,u^2)$ coincide at
time $t=0,$ so do $v^1$ and $v^2$. Hence, uniqueness for Burgers equation implies $v^1=v^2:=v$. We thus have  
$$
\left\{\begin{array}{l} 
(\d_t+w^2\cdot\nabla)\dr+\frac{\gamma-1}2\rho^2\,\div\dw+v\cdot\nabla\dr+\frac{\gamma-1}2\dr\,\div v=
-\dw\cdot\nabla\rho^1-\frac{\gamma-1}2\dr\,\div w^1,\\[1.5ex]
(\d_t+w^2\cdot\nabla)\dw+\frac{\gamma-1}2\rho^2\,\nabla\dr+v\cdot\nabla\dw+\dw\cdot\nabla v\\
\hspace{6cm}=-\dw\cdot\nabla w^1
         -\frac{\gamma-1}2\dr\nabla\rho^1+\kappa\nabla\df,\\[1.5ex]
\Delta\df-\mu^2\df=\wt G \bigl((\rho^2)^{\frac2{\gamma-1}}-(\rho^1)^{\frac2{\gamma-1}}\bigr)\cdotp\end{array}\right.
$$
Taking the $L^2$ scalar product of the first and second equations with $\dr$ and $\dw,$ respectively,
and arguing as for proving  \eqref{eq:L20}, we get 
$$\displaylines{\frac12\frac d{dt}\|(\dr,\dw)\|_{L^2}^2=\frac12\int_{\R^d}(\dr^2+|\dw|^2)\div(v+w^2)\,dx
+\frac{\gamma-1}2\!\int_{\R^d}\dr\,\dw\cdot\nabla(\rho^2-\rho^1)\,dx\hfill\cr\hfill
-\int_{\R^d}\dw\cdot\nabla\rho^1\dr\,dx-\frac{\gamma-1}2\int_{\R^d}(\dr)^2\div(v+w^1)\,dx
\hfill\cr\hfill-\int_{\R^d}(\dw\cdot\nabla(v+w^1))\cdot\dw\,dx +\kappa\int_{\R^d}\nabla\df\cdot\dw\,dx,}$$
whence 
$$
\frac d{dt}\|(\dr,\dw)\|_{L^2}^2\leq C\|\nabla v,\nabla w^1,\nabla w^2,\nabla\rho^1,\nabla\rho^2\|_{L^\infty}
\|(\dr,\dw)\|_{L^2}^2+|\kappa|\|\nabla\df\|_{L^2}\|\dw\|_{L^2}.
$$
After time integration, this gives  for all $t\in[0,T]$ (since $\dr|_{t=0}=0$ and $\dw|_{t=0}=0$): 
$$\|(\dr,\dw)(t)\|_{L^2}\leq C\int_0^t\|\nabla v,\nabla w^1,\nabla w^2,\nabla\rho^1,\nabla\rho^2\|_{L^\infty}
\|(\dr,\dw)\|_{L^2}\,d\tau +|\kappa|\int_0^t \|\nabla\df\|_{L^2}.$$
\begin{itemize}
\item[---] In the pure Euler case (namely $\kappa=0$), applying Gronwall lemma
readily gives $(\dr,\dw)\equiv(0,0)$ on $[0,T].$
\item[---] In the Poisson case  ($\kappa\not=0$ and $\mu=0$)
with  $d\geq3$ one can use the fact that operator $\nabla(-\Delta)^{-1}$ maps 
$L^p$ (with $\frac1p=\frac1d+\frac12$) to $L^2$ and that
$$
(\rho^2)^{\frac2{\gamma-1}}-(\rho^1)^{\frac2{\gamma-1}}=\frac2{\gamma-1}\dr\int_0^1 (\rho^1+\eta\dr)^{\frac2{\gamma-1}-1}\,d\eta.
$$
Hence 
$$\|\nabla\df\|_{L^2}\leq C\|\dr\|_{L^2}\biggl\|\int_0^1 (\rho^1+\eta\dr)^{\frac2{\gamma-1}-1}\,d\eta\biggr\|_{L^d}\cdotp$$
Note that since $\gamma\leq 1+2d/(d+2),$ and $\rho^i\in L^\infty(0,T;L^2\cap L^\infty)$ for $i=1,2,$ 
we are guaranteed  that the integral is  bounded  in terms of $\rho^1$ and $\rho^2.$ 
Hence, we have 
$$\|\nabla\df\|_{L^2}\leq C_{\rho^1,\rho^2}\|\dr\|_{L^2}$$
and applying Gronwall lemma still ensures  uniqueness.
\item[---] In the Helmholtz  case ($\kappa\not=0$ and $\mu\not=0$) with  $d\geq2,$
we just use the fact that $\nabla(\mu-\Delta)^{-1}$ maps $L^2$ to itself, and thus, arguing 
as above,  
 $$\|\nabla\df\|_{L^2}\leq
 C\|\dr\|_{L^2}\biggl\|\int_0^1 (\rho^1+\eta\dr)^{\frac2{\gamma-1}-1}\,d\eta\biggr\|_{L^\infty}\cdotp$$
Since $\rho^1$ and $\rho^2$ are bounded, one can apply Gronwall lemma to 
prove that the two solutions coincide. 
\end{itemize}
This completes the proof of the proposition. \qed


\vskip1cm
\section*{Appendix}
\vskip1cm

For the reader's convenience, we prove here some technical results that have been used in the paper. 
Let us start with an ODE estimate.
\begin{lem}\label{l:ODE} Let $Y:\R_+\to\R_+$ satisfy the differential inequality 
$$\frac d{dt}Y+\frac a{1+t}Y\leq C\biggl(\frac Y{(1+t)^2}+Y^2+(1+t)^{m'-1}Y^{m+1}\biggr)\quad\hbox{on }\ \R_+$$
for some $C>0,$ $a>1,$ $m>0$ and $m'<ma.$ 
Then, there exists $c=c(a,m,m',C)$ such that if $Y(0)\leq c,$ then we have
$$
Y(t)\leq 2e^{\frac{Ct}{1+t}}\frac{Y_0}{(1+t)^a}\quad\hbox{for all }\ t\geq0.
$$
\end{lem}
\begin{proof}
We set $Z(t):=(1+t)^ae^{-\frac{Ct}{1+t}}Y(t)$ and observe that  the above differential inequality recasts in 
$$\frac d{dt} Z\leq C(1+t)^{-a}e^{\frac{Ct}{1+t}}Z^2+C(1+t)^{m'-ma-1}e^{\frac{Cmt}{1+t}}Z^{m+1},$$
which implies that
\begin{equation}\label{eq:Z1}\frac d{dt} Z\leq Ce^C(1+t)^{-a}Z^2+C(1+t)^{m'-ma-1}e^{Cm}Z^{m+1}.\end{equation}
The conclusion stems from a  bootstrap argument : let $Z_0:=Z(0)$ and  assume that 
\begin{equation}\label{eq:Z2}
Z(t)\leq 2Z_0\quad\hbox{on}\quad [0,T].\end{equation}
 Then \eqref{eq:Z1} implies that
$$\frac d{dt} Z\leq 4Ce^C(1+t)^{-a}Z_0^2+C(1+t)^{m'-ma-1}e^{Cm}(2Z_0)^{m+1}.$$
Hence, integrating in time, we discover that on $[0,T],$ we have
$$Z(t)\leq Z_0+\frac{4Ce^C}{a-1}Z_0^2\bigl(1-(1+t)^{1-a}\bigr)+\frac{Ce^{Cm}(2Z_0)^{m+1}}{ma-m'}\bigl(1-(1+t)^{m'-ma}\bigr)\cdotp$$
Let us discard the obvious case $Z_0=0.$ Then, if  $Z_0$ is so small as to satisfy 
$$\frac{4Ce^C}{a-1}Z_0+\frac{2^{m+1}Ce^{Cm} Z_0^m}{ma-m'}\leq 1,$$
the above inequality ensures that  we actually have $Z(t)<2Z_0$ on $[0,T]$. Therefore the supremum of $T>0$ satisfying \eqref{eq:Z2} has to be infinite.
\end{proof}

The following result has been used to bound the potential term. 
\begin{lem}\label{l:compo} Let $\alpha\geq1$ and $0\leq\sigma<\alpha+\frac12\cdotp$ 
Then we have the inequality
\begin{equation}\label{eq:compo}
\||z|^\alpha\|_{\dot H^\sigma}\lesssim \|z\|_{L^\infty}^{\alpha-1}\|z\|_{\dot H^\sigma}.
\end{equation}
\end{lem}
\begin{proof}
The corresponding  inequality  for nonhomogeneous Sobolev spaces, namely
$$\||z|^\alpha\|_{H^\sigma}\lesssim \|z\|_{L^\infty}^{\alpha-1}\|z\|_{H^\sigma}$$
  is a particular case of \cite[Thm. 1.2]{Kateb}.
\medbreak
Applying that inequality to $z(\lambda\cdot)$ for all $\lambda>0,$ and using the fact that
$$
\||z(\lambda\cdot)|^\alpha\|_{\dot H^\sigma}=\lambda^{\sigma-\frac d2} 
\||z|^\alpha\|_{\dot H^\sigma}\andf
\|z(\lambda\cdot)\|_{H^\sigma}\approx \lambda^{-\frac d2}\bigl(\|z\|_{L^2} + \lambda^{\sigma} \|z\|_{\dot H^\sigma}\bigr),
$$
we get the desired inequality after having $\lambda$ tend to $+\infty.$
\end{proof}


We also used the following first order  commutator estimate
that  corresponds to  the end of \cite[Rem. 1.5]{DL}, or may be seen as a straightforward adaptation 
  to the homogeneous framework of the second inequality of
\cite[Lemma A.2]{BDD}:
\begin{lem}\label{l:com1} If $s>0,$ then we have:
$$\|[v,\dot\Lambda^s]u\|_{L^2}\lesssim
\|v\|_{\dot H^s}\|u\|_{L^\infty}+\|\nabla v\|_{L^\infty}\|u\|_{\dot H^{s-1}}.$$
\end{lem}

The following second order commutator inequality  played a key role in the proof 
of Sobolev estimates with noninteger exponent for the solution to $(BB).$ 
\begin{lem}\label{l:com2} If $s>1,$ then we have:
$$\|[v,\dot\Lambda^s]u-s\nabla v\cdot\dot\Lambda^{s-2}\nabla u\|_{L^2}\lesssim
\|v\|_{\dot H^s}\|u\|_{L^\infty}+\|\nabla^2v\|_{L^\infty}\|u\|_{\dot H^{s-2}}.$$
\end{lem}
\begin{proof}
  The proof follows the lines of that of \cite[Lemma A.3]{BDD}, which deals with the non-homogeneous case. It relies on  Bony's decomposition and on continuity results for 
the paraproduct and remainder operators. 
For the reader convenience, let us shortly recall how it works. 
Fix some smooth radial function $\chi$ supported in (say) the ball
$B(0,4/3)$ and with value $1$ on $B(0,3/4),$ then set $\varphi:=\chi(\cdot/2)-\chi.$
For all $j\in\Z,$ we define the spectral cut-off operators
$\ddj$  and $\dot S_j$ acting on tempered distributions $u$ as follows:
$$\begin{aligned}
\ddj u&:=\cF^{-1}(\varphi(2^{-j}\cdot)\cF u)=2^{jd} (\cF^{-1}\varphi)(2^j\cdot)\star u,\\
\dot S_j u&:=\cF^{-1}(\chi(2^{-j}\cdot)\cF u)=2^{jd} (\cF^{-1}\chi)(2^j\cdot)\star u,
\end{aligned}
$$
where $\cF$ denotes the Fourier transform on $\R^d.$
\medbreak
Whenever the product $uv$ of two tempered distributions $u$ and $v$
is defined,  its so-called Bony's decomposition  (first introduced in \cite{Bony}) reads:
$$
uv=T_uv+T_vu+R(u,v)
$$
where the \emph{paraproduct} operator $T$ and \emph{remainder} operator $R$ are defined by 
$$
T_uv:=\sum_{j\in\Z}\dot S_{j-1}u\,\ddj v\andf
R(u,v):=\sum_{j\in\Z} \ddj u\,\bigl(\dot\Delta_{j-1}v+\ddj v+\dot\Delta_{j+1}v\bigr)\cdotp
$$
One can now start the proof. 
Decomposing the terms $v\dot \Lambda^s u,$ $\dot\Lambda^s(uv)$ and $\nabla v\cdot\dot\Lambda^{s-2}\nabla u$
according to  Bony's decomposition, we get with the usual summation convention, 
$$
\displaylines{
[v,\dot \Lambda^s]u-  s\nabla v\cdot \dot\Lambda^{s-2}\nabla u=
\underbrace{[T_v,\dot\Lambda^s]u-s\,T_{\d_kv}\d_k\dot\Lambda^{s-2}u}_{R^1}
+\underbrace{T_{\dot\Lambda^su}v}_{R^2}
-\underbrace{\dot\Lambda^sT_uv}_{R^3}\hfill\cr\hfill
-s\underbrace{T_{\d_k\dot\Lambda^{s-2}u}\d_kv}_{R^4}
+\underbrace{R(v,\dot\Lambda^s u)}_{R^5}
-\underbrace{\dot\Lambda^sR(v,u)}_{R^6}
-s\underbrace{R(\d_kv,\d_k\dot\Lambda^{s-2}u)}_{R^7}.}
$$


\subsubsection*{Bounding $R^1$}

{}From the definition of the paraproduct, we have
$R^1=\sum_{j\in\Z} R^1_j$
with  
$$
R^1_j:=\dot S_{j-1}v\,\tilde\varphi(2^{-j}{\rm D})\dot\Lambda^s\,\ddj u
-\tilde\varphi(2^{-j}{\rm D})\dot\Lambda^s\bigl(
\dot S_{j-1}v\ddj u\bigr)-s\dot S_{j-1}\d_kv\,
\tilde\varphi(2^{-j}{\rm D})(\d^k\dot\Lambda^{s-2})\ddj u,
$$
for some suitable  $\tilde\varphi\in\cC_c^\infty(\R^d)$ supported  in an annulus. 
{}From second order Taylor's formula, 
we gather  
$$
R_j^1=
-\Int_{\R^d}\!\Int_0^1 h_{s,j}(y)
y\cdot D^2\dot S_{j-1}v(x-\tau y)\cdot y\:\ddj u(x-y)(1-\tau)\,d\tau\,dy,
$$
where $h_{s,j}:={\cF}^{-1}(|\cdot|^s\tilde\varphi(2^{-j}\cdot))=2^{js}\cF^{-1}(\wt\varphi_s(2^{-j}\cdot))$
and $\wt\varphi_s(\xi):=|\xi|^s\wt\varphi(\xi).$ 
\medbreak
Therefore, 
$$\begin{aligned}\|{R_j^1}\|_{L^2}&\leq2^{j(s-2)}\||\cdot|^2\cF^{-1}\wt\varphi_s\|_{L^1}\|D^2\dot S_{j-1}v\|_{L^\infty}\|\ddj u\|_{L^2}\\
&\lesssim \|D^2v\|_{L^\infty} 2^{j(s-2)} \|\ddj u\|_{L^2}.\end{aligned}$$
Since the spectral localization of the terms $R_j^1$ implies that
$$
\|R^1\|_{L^2}^2\approx \sum_{j\in\Z} \|R_j^1\|_{L^2}^2,
$$
and because, for any $\sigma\in\R,$ we have
$$
\|z\|_{\dot H^\sigma}^2\approx \sum_{j\in\Z} 2^{2j\sigma} \|\ddj z\|_{L^2}^2,
$$
one ends up with 
$$
\|R^1\|_{L^2}\lesssim\|\nabla^2 v\|_{L^\infty}\|u\|_{\dot H^{s-2}}.
$$


\subsubsection*{Bounding $R^2$}
Combining a standard continuity result  for the paraproduct (see e.g. \cite[Chap. 2]{BCD}) 
with the fact that $\dot\Lambda^s:L^\infty\to\dot B^{-s}_{\infty,\infty}$ is continuous implies if $s>0,$
$$
\|R^2\|_{L^2}\lesssim \|\dot\Lambda^su\|_{\dot B^{-s}_{\infty,\infty}}\|v\|_{\dot H^{s}}\lesssim \|u\|_{L^\infty}\|v\|_{\dot H^{s}}.
$$


\subsubsection*{Bounding $R^3$}
Since $\dot\Lambda^s$ maps ${\dot H}^{s}$ to $L^2$, 
we have for all $s\in\R$, 
$$
\|R^3\|_{L^2}\lesssim\|u\|_{L^\infty}\|v\|_{\dot H^{s}}.
$$


\subsubsection*{Bounding $R^4$}
Standard continuity results for the paraproduct
combined with the fact that 
 $(\d^k|\cdot|^s)({\rm D})$ is a homogeneous multiplier  of degree $s-1$
yield, if $s>1,$  
$$
 \|R^4\|_{L^2}\lesssim\|\nabla\dot\Lambda^{s-2}u\|_{\dot B^{1-s}_{\infty,\infty}}
\|{\nabla v}\|_{{\dot H}^{s-1}} \lesssim\|u\|_{L^\infty}\|v\|_{{\dot H}^{s}}.
$$


\subsubsection*{Bounding $R^5$}
Basic results of continuity for
the remainder, see e.g. \cite{RS}, ensure that for  $s\in\R,$ since $\dot\Lambda^s$ maps 
${\rm L}^\infty$ to  ${\dot F}^{-s}_{\infty,2}$, 
$$
\|R^5\|_{L^2}\lesssim\|v\|_{\dot H^{s}}\|\dot\Lambda^s u\|_{\dot F^{-s}_{\infty,2}}\lesssim\|v\|_{\dot H^{s}}\|u\|_{L^\infty}.
$$


\subsubsection*{Bounding $R^6$}
If  $s>0$, then we have
$$\|R^6\|_{L^2}\lesssim
\|v\|_{{\dot H}^{s}}\|u\|_{L^\infty}, 
$$


\subsubsection*{Bounding $R^7$}
 Since $\d^k\dot\Lambda^{s-2}$ is a homogeneous  multiplier
of degree $s\!-\!1$, we  have
$$\|R^7\|_{L^2}\lesssim\|\nabla v\|_{\dot H^{s-1}}\|\d^k\dot\Lambda^{s-2}u\|_{\dot F^{1-s}_{\infty,2}}
\lesssim\|\nabla v\|_{\dot H^{s-1}}\|u\|_{L^\infty}. $$
Putting together all the above estimates gives what we want.
\end{proof}

\subsection* {Acknowledgments:}
We warmly thank Denis Serre \cite{Scom} for his interest in our work and fruitful discussions about his results in \cite{S}, \cite{GS}, \cite{G2} and the doctoral thesis of Magali Grassin \cite{G}.

 \v S\' arka Ne\v casov\'a was supported by the Czech Science Foundation grant GA19-04243S in the framework of RVO 67985840.

  Rapha\"el Danchin and Bernard Ducomet are partially supported by the ANR project INFAMIE (ANR-15-CE40-0011).

\vskip1cm
\centerline{Xavier Blanc}
 \centerline{Universit\'e de Paris, Laboratoire Jacques-Louis Lions (LJLL), F-75005 Paris, France}
\centerline{Sorbonne Universit\'e, CNRS, LJLL, F-75005 Paris, France}
 \centerline{E-mail: blanc@ann.jussieu.fr}
\vskip0.5cm
\centerline{Rapha\"el Danchin}
 \centerline{Universit\'e Paris-Est}
\centerline{LAMA (UMR 8050), UPEMLV, UPEC, CNRS}
\centerline{ 61 Avenue du G\'en\'eral de Gaulle, F-94010 Cr\'eteil, France}
 \centerline{E-mail: danchin@u-pec.fr}
\vskip0.5cm
\centerline{Bernard Ducomet}
 \centerline{Universit\'e Paris-Est}
\centerline{LAMA (UMR 8050), UPEMLV, UPEC, CNRS}
\centerline{ 61 Avenue du G\'en\'eral de Gaulle, F-94010 Cr\'eteil, France}
 \centerline{E-mail: bernard.ducomet@u-pec.fr}
\vskip0.5cm
\centerline{\v S\' arka Ne\v casov\' a}
\centerline{Institute of Mathematics of the Academy of Sciences of the Czech Republic}
\centerline{\v Zitn\' a 25, 115 67 Praha 1, Czech Republic}
\centerline{E-mail: matus@math.cas.cz}


\begin{thebibliography}{99}
 \bibitem{BCD} 
 H. Bahouri, J.-Y. Chemin and  R. Danchin.
\newblock  {\it Fourier Analysis and Nonlinear Partial Differential Equations.} 
\newblock Grundlehren der mathematischen Wissenschaften, {\bf 343},  Springer (2011).

\bibitem{BDD}
 S. Benzoni-Gavage, R. Danchin and S. Descombes.
\newblock  On the well-posedness for the Euler-Korteweg model in several space dimensions. 
\newblock {\em Indiana Univ. Math. J.}, {\bf 56}(4): 1499--1579, 2007. 

\bibitem{BL1}
H. Berestycki and P.-L. Lions.
\newblock Existence of solutions for nonlinear scalar field equations, I Existence of a ground state.
\newblock {\em Arch. Ration. Mech. Anal.}, {\bf 82}:313--346, 1983.

\bibitem{BL2}
H. Berestycki and P.-L. Lions.
\newblock Existence of solutions for nonlinear scalar field equations, II Existence of infinitely many solutions.
\newblock {\em Arch. Ration. Mech. Anal.}, {\bf 82}:347--376, 1983.

\bibitem{be} M. B\'ezard.
\newblock Existence locale de solutions pour les equations d'Euler-Poisson.
\newblock {\em Japan J. Indust. Appl. Math.}, {\bf 10}:431--450, 1993.

\bibitem{Bony} J.-M. Bony. 
\newblock Calcul symbolique et propagation des singularit\'es pour les \'equations 
aux d\'eriv\'ees partielles non lin\'eaires.
\newblock{\em Ann. Sci. Ecole Norm. Sup.}, {\bf 14}:209--246, 1981.

\bibitem{BK} U. Brauer and L. Karp.
\newblock Local existence of solutions to the Euler-Poisson system including densities without compact support.
\newblock {\em Journal of Differential Equations}, {\bf 264}:755--785, 2018.

\bibitem{Caz}
T. Cazenave.
\newblock {\em An introduction to semilinear elliptic equations}.
\newblock Editions of IM-UFRJ, Rio de Janeiro, 2006.

\bibitem{Ch}
S. Chandrasekhar.
\newblock {\em An introduction to the study of stellar structure}.
\newblock Dover Publications, New York, 1957.

\bibitem{CC}
S.G. Chefranov and A.S. Chefranov.
\newblock Exact time-dependent solution to the three-dimensional Euler-Helmholtz and Riemann-Hopf equations for vortex flow of a compressible medium and one of
the millenium prize problems.
\newblock {\em arXiv:1703.07239v3}, 25 Sep 2017.

\bibitem{C}
J.M. Chemin.
\newblock Dynamique des gaz \`a masse totale finie.
\newblock {\em Asymptotic Analysis}, {\bf 3}:215--220, 1990.

\bibitem{CW}
G.Q. Chen and  D. Wang.
\newblock The Cauchy problem for the Euler equations for compressible fluids.
\newblock {\em in ``Handbook of Mathematical Fluid Dynamics, Vol. 1", S. Friedlander, D. Serre Eds.}
\newblock North-Holland, Elsevier, Amsterdam, Boston, London, New York, 2002.

\bibitem{Chi}
H-Y. Chiu.
\newblock {\em Stellar physics}.
\newblock Blaisdell Publishing Company, Waltham, Toronto, London, 1968.

\bibitem{CH}
A.R. Choudhuri.
\newblock {\em The physics of fluids and plasmas. An introduction for astrophysicists}.
\newblock Cambridge University Press, 1998.


\bibitem{PDA}
P. d'Ancona.
\newblock A short proof of commutator estimates.
\newblock  {\em arXiv:1709.01294v2 [math.AP]}, 14 Mar 2018.

\bibitem{EL}
M.J. Esteban and  P.-L. Lions.
\newblock Existence and non-existence results for semilinear elliptic problems in unbounded domains.
\newblock {\em Proc. Roy. Soc. of Edinburgh}, {\bf 93A}:1--14, 1982.


\bibitem{ga}
P. Gamblin.
\newblock Solution r\'eguli\`ere \`a temps petit pour l'\'equation d'Euler-Poisson.
\newblock {\em Commun. in Partial Differential Equations}, {\bf 18}:731--745, 1993.

\bibitem{GNN}
B. Gidas, W-M. Ni and  L. Nirenberg.
\newblock Symmetry and related properties via the maximun principle.
\newblock {\em Commun. Math. Phys.}, {\bf 68}:209--243, 1979.

\bibitem{G}
M. Grassin-Hillairet.
\newblock Existence et stabilit\'e de solutions globales en dynamique des gaz.
\newblock {\em PHD thesis, Ecole Normale Sup\'erieure de Lyon}, 1999.

\bibitem{GS}
M. Grassin and D. Serre.
\newblock Existence de solutions globales et r\'eguli\`eres aux \'equations d'Euler pour un gaz parfait isentropique.
\newblock {\em C.R. Acad. Sci. Paris, S\'erie I}, {\bf 325}:721--726, 1997.

\bibitem{G2}
 M. Grassin.
\newblock Global smooth solutions to Euler equations for a perfect gas.
\newblock {\em Indiana Univ. Math. J.}, {\bf 47}:1397--1432, 1998.

\bibitem{Gu}
Y. Guo.
\newblock Smooth irrotational flows in the large to the Euler-Poisson system in $\R^{3+1}$.
\newblock {\em Commun. Math. Phys.}, {\bf 195}:249--265, 1998.


\bibitem{J}
J. Jang.
\newblock Nonlinear instability in gravitational Euler-Poisson system for $\gamma=6/5$.
\newblock {\em Arch. Ration. Mech. Anal.}, {\bf 188}:265--307, 2008.


\bibitem{Kateb} D. Kateb.
 \newblock On the boundedness of the mapping $f\mapsto |f|^\mu,$ $\mu>1$  on Besov spaces. 
 \newblock {\em Math. Nachr.}, {\bf  248/249}:110--128 (2003). 


\bibitem{K}
T. Kato.
\newblock The Cauchy problem for quasi-linear symmetric hyperbolic systems.
\newblock {\em Arch. Ration. Mech. Anal.}, {\bf 58}:181--205, 1975.

\bibitem{KP}
 T. Kato and G. Ponce.
  Commutator estimates and the Euler and Navier-Stokes equations. 
{\em Comm. Pure Appl. Math.}, {\bf 41}(7):  891--907 (1988). 

\bibitem{KPV}
 C.E. Kenig, G. Ponce and  L. Vega.
 Well-posedness and scattering results for the generalized Korteweg-de-Vries equation via the contraction principle,
{\em Comm. Pure Appl. Math.}, {\bf 46}(4):  527--620 (1993). 

\bibitem{LM}
M. L\'ecureux-Mercier.
\newblock Global smooth solutions of Euler equations for Van der Waals gases.
\newblock {\em SIAM J. Math. Anal.}, {\bf 43}:877--903, 2011.

\bibitem{DL}
D. Li.
\newblock On Kato-Ponce and fractional Leibniz.
\newblock  {\em Revista Matematica Iberoamericana}, {\bf 35}(1):23--100, 2019.


\bibitem{Li}
S-S. Lin.
\newblock Stability of gaseous stars in spherically symmetric motions.
\newblock {\em SIAM J. Math. Anal.}, {\bf 28}:539--569, 1997.

\bibitem{Lio}
P.-L. Lions.
\newblock Minimization problems in $L^1({\mathbb R}^3)$.
\newblock {\em Journal of Functional Analysis}, {\bf 41}:236--275, 1981.

\bibitem{M}
A. Majda.
\newblock {\em Compressible fluid flow and systems of conservation laws in several variables}.
\newblock Springer-Verlag, New-York, Berlin, Heidelberg, Tokyo, 1984.

\bibitem{Ma}
T. Makino.
\newblock On a local existence theorem for the evolution equation of gaseous stars.
\newblock In {\em Patterns and Waves-Qualitative Analysis of Nonlinear Differential Equations}, {\bf 3}:459--479, 1986.

\bibitem{Ma2}
T. Makino.
\newblock Blowing-up solutions of the Euler-Poisson equations for the evolution of gaseous stars.
\newblock \newblock {\em Transport Theory and Statistical Physics}, {\bf 21}:615--624, 1992.

\bibitem{Ma3}
T. Makino.
\newblock Mathematical aspects of the Euler-Poisson equations for the evolution of gaseous stars.
\newblock \newblock {\em NCTU-MATH 930001}, 
\newblock {\em Lect. Notes Dep. of Applied Math., National Chiao Tung University, Taiwan, R.O.C.}, March 2003.


\bibitem{MP}
T. Makino and  B. Perthame.
\newblock Sur les solutions \`a sym\'etrie sph\'erique de l'\'equation d'Euler-Poisson pour l'\'evolution d'\'etoiles gazeuses.
\newblock {\em Japan J. Appl. Math.}, {\bf 7}:165--170, 1990.

\bibitem{MU}
T. Makino and  S. Ukai.
\newblock Sur l'existence des solutions locales de l'\'equation d'Euler-Poisson pour l'\'evolution d'\'etoiles gazeuses.
\newblock {\em  J. Math. Kyoto Univ}, {\bf 27}:387--399, 1987.




\bibitem{P}
B. Perthame.
\newblock Non-existence of global solutions to Euler-Poisson equations for repulsive forces.
\newblock {\em Japan J. Appl. Math.}, {\bf 7}:363--367, 1990.


\bibitem{R}
R. Racke.
\newblock {\em Lectures on nonlinear evolution equations}.
\newblock Vieweg, Braunschweig, Wiesbaden, 1992.

\bibitem{Re}
G. Rein.
\newblock Nonlinear stability of gaseous stars.
\newblock {\em Arch. Ration. Mech. Anal.}, {\bf 168}:115--130, 2003.


\bibitem{RS} T. Runst and W. Sickel.
{\em Sobolev spaces of fractional order, Nemytskij operators, and nonlinear
partial differential equations}, Nonlinear Analysis and Applications, {\bf 3}. Walter de Gruyter \& Co., Berlin, 1996.

\bibitem{SP}
E. Schatzman and F. Praderie.
\newblock {\em Les \'etoiles}.
\newblock InterEditions, Editions du CNRS, Paris, 1990.

\bibitem{S}
 D. Serre.
\newblock Solutions classiques globales des \'equations d'Euler pour un fluide parfait compressible.
\newblock {\em Ann. Inst. Fourier, Grenoble}, {\bf 47}:139--159, 1997.

\bibitem{Scom}
D. Serre.
\newblock {\em Personal communications}.


\end{thebibliography}
\end{document}